\documentclass[11pt,reqno]{amsart}
\usepackage[top=1.2in, bottom=1.2in, left=1.5in, right=1.5in]{geometry}

\usepackage{amsfonts, amssymb, amsmath, enumerate}

\numberwithin{equation}{section}

\DeclareMathOperator{\E}{\mathbb{E}}

\DeclareMathOperator*{\Span}{span}

\DeclareMathOperator*{\tr}{tr}
\DeclareMathOperator*{\rank}{rank}

\renewcommand{\Pr}[2][]{\mathbb{P}_{#1} \left\{ #2 \rule{0mm}{3mm}\right\}}

\def \N {\mathbb{N}}
\def \P {\mathbb{P}}
\def \R {\mathbb{R}}

\def \LL {\mathcal{L}}

\def \e {\varepsilon}
\def \d {\delta}

\def \HS {\mathrm{HS}}

\def \tran {\mathsf{T}}
\def \HS {\mathrm{HS}}
\def \< {\langle}
\def \> {\rangle}
\def \etc {,\ldots,}

\newcommand{\norm}[1]{\left \| #1 \right \|}

\newtheorem{theorem}{Theorem}[section]
\newtheorem{proposition}[theorem]{Proposition}
\newtheorem{corollary}[theorem]{Corollary}
\newtheorem{lemma}[theorem]{Lemma}

\theoremstyle{remark}
\newtheorem{remark}[theorem]{Remark}

\title[]{Small ball probabilities for linear images of high dimensional distributions}

\author{Mark Rudelson}
\author{Roman Vershynin}

\date{\today}
\address{Department of Mathematics, University of Michigan, 530 Church St., Ann Arbor, MI 48109, U.S.A.}
\email{\{rudelson, romanv\}@umich.edu}
\thanks{Partially supported by NSF grants DMS 1161372, 1001829, 1265782 and USAF Grant FA9550-14-1-0009.}

\begin{document}

\begin{abstract}
  We study concentration properties of random vectors of the form $AX$,
  where $X = (X_1,\ldots, X_n)$ has independent coordinates and $A$ is a given matrix. 
  We show that the distribution of $AX$ is well spread in space whenever the distributions of $X_i$ 
  are well spread on the line. 
  Specifically, assume that the probability that $X_i$ 
  falls in any given interval of length $t$ is at most $p$. Then 
  the probability that $AX$ falls in any given ball of radius $t \|A\|_\HS$ is at most $(Cp)^{0.9 \, r(A)}$,
  where $r(A)$ denotes the stable rank of $A$ and $C$ is an absolute constant.
\end{abstract}

\maketitle

\section{Introduction}

Concentration properties of high dimensional distributions have been extensively studied in probability theory. 
In this paper we are interested in {\em small ball 
probabilities}, which describe the spread of a distribution in space.
Small ball probabilities have been extensively studied for stochastic processes (see \cite{Li-Shao}),
sums of independent random variables (see \cite{TV, RV ICM}) and
log-concave measures (see \cite[Chapter 5]{ABVV}).
Nevertheless, there remain surprisingly basic questions that have not been previously addressed. 

The main object of our study is a random vector $X = (X_1,\ldots, X_n)$ in $\R^n$
with independent coordinates $X_i$. Given a fixed $m \times n$ matrix $A$, 
we study the concentration properties of the random vector $AX$. We are interested in results of the 
following type:

\begin{quote} 
{\em If the distributions of $X_i$ are well spread on the line, then 
the distribution of $AX$ is well spread in space.} 
\end{quote}

Special cases of interest are {\em marginals} of $X$ which arise when $A$ is an orthogonal 
projection, and {\em sums of independent random variables} which correspond to $m=1$.
The problem of describing small ball probabilities even in these two special cases is nontrivial
and useful for applications. 
In particular, a recent interest in this problem was spurred by applications in random matrix theory;
see \cite{TV, RV ICM} for sums of random variables and \cite{RV rectangular} for higher dimensional 
marginals.

This discussion is non-trivial even for continuous distributions, and we shall start 
from this special case.

\subsection{Continuous distributions}

Our first main result is about distributions with 
independent continuous coordinates. It states that any $d$-dimensional 
marginal of such distribution has density bounded by $O(1)^d$, 
as long as the densities of the coordinates are bounded by $O(1)$.

\begin{theorem}[Densities of projections]					\label{thm: dens PX}
  Let $X= (X_1,\ldots, X_n)$ where
  $X_i$ are real-valued independent random variables.
  Assume that the densities of $X_i$ are bounded by $K$
  almost everywhere.
  Let $P$ be the orthogonal projection in $\R^n$ onto a $d$-dimensional subspace.
  Then the density of the random vector $PX$ is bounded by $(CK)^d$ almost everywhere.
\end{theorem}

Here and throughout the paper, $C, C_1, c, c_1, \ldots$ denote positive absolute constants.

\medskip

Theorem~\ref{thm: dens PX} is trivial in dimension $d=n$, since the product density
is bounded by $K^n$. A remarkable non-trivial case of Theorem~\ref{thm: dens PX}
is in dimension $d=1$, where it holds with optimal constant $C=\sqrt{2}$.
This partial case is worth to be stated separately.

\begin{theorem}[Densities of sums]					\label{thm: dens sums}
  Let $X_1,\ldots, X_n$ be real-valued independent random variables
  whose densities are bounded by $K$ almost everywhere.
  Let $a_1,\ldots,a_n$ be real numbers with $\sum_{j=1}^n a_j^2 = 1$.
  Then the density of $\sum_{j=1}^n a_j X_j$ is bounded by
  $\sqrt{2}\, K$ almost everywhere.

  Moreover, this estimate is optimal. The density $\sqrt{2}\, K$ is achieved
  by the sum $\frac{1}{\sqrt{2}} (X_1 + X_2)$ where $X_1, X_2$ are uniformly
  distributed in $[-\frac{1}{2K}, \frac{1}{2K}]$.
\end{theorem}

Theorem~\ref{thm: dens sums} follows immediately from a combination of two known results, 
a theorem of B.~Rogozin stated below and a theorem of K.~Ball \cite{Ball 86}.

\begin{theorem}[Rogozin \cite{Rogozin}]					\label{thm: rogozin}
  Let $Z_1,\ldots,Z_n$ be real-valued independent random variables 
  whose densities are bounded by $d_1,\ldots,d_n$ respectively. 
  Then the density of $\sum_{j=1}^n Z_j$ is uniformly bounded by
  the value of the density of $\sum_{j=1}^n U_j$ at the origin, 
  where $U_j$ are independent random variables uniformly distributed 
  in $[-\frac{1}{2d_j}, \frac{1}{2d_j}]$.
\end{theorem}

\begin{proof}[Proof of Theorem~\ref{thm: dens sums}]
For simplicity, by rescaling we can assume that $K=1$.
Rogozin's Theorem~\ref{thm: rogozin} implies that the 
density of $\sum_{j=1}^n a_j X_j$ is uniformly bounded by the density
of $\sum_{j=1}^n a_j Y_j$ at the origin, where $Y_j$ are independent random variables 
uniformly distributed in $[-\frac{1}{2}, \frac{1}{2}]$.

The random vector $Y = (Y_1,\ldots,Y_n)$ is uniformly distributed in the cube $[-\frac{1}{2},\frac{1}{2}]^n$.
Thus the density of $\sum_{j=1}^n a_j Y_j$ at the origin is the volume 
of the section of the cube by the hyperplane that contains the origin and is orthogonal to 
the vector $a = (a_1,\ldots, a_n)$.
Now, a theorem of Ball \cite{Ball 86}
states that the maximal volume of such section equals $\sqrt{2}$.
Therefore, the density of $\sum_{j=1}^n a_j X_j$ is bounded by $\sqrt{2}$.
\end{proof}

\subsection{General distributions}

Using a simple smoothing argument from \cite{BN}, 
Theorem~\ref{thm: dens PX} can be extended for general,
not necessarily continuous, distributions. The spread 
of general distributions is conveniently measured by the {\em concentration function}.
For a random vector $Z$ taking values in
a $\R^n$, the concentration function is defined as
\begin{equation}         \label{eq: concentration function}
\LL(Z,t) = \max_{u \in \R^n} \Pr{ \|Z-u\|_2 \le t }, \quad t \ge 0.
\end{equation}
Thus the concentration function controls the {\em small ball probabilities} of 
the distribution of $Z$.
The study of concentration functions of sums of independent random variables
originates from the works of L\'evy~\cite{Levy}, Kolmogorov~\cite{Kolm}, Rogozin~\cite{Rogozin 61}, 
Esseen~\cite{Esseen} and Halasz~\cite{Halasz 77}.
Recent developments in this area highlighted connections with 
Littlewood-Offord problem and applications to random matrix theory, see \cite{TV, RV ICM}.

\begin{corollary}[Concentration function of projections]			\label{cor: conc PX}
  Consider a random vector $X= (X_1,\ldots, X_n)$ where
  $X_i$ are real-valued independent random variables.
  Let $t,p \ge 0$ be such that
  $$
  \LL(X_i, t) \le p \quad \text{for all } i=1,\ldots,n.
  $$ 
  Let $P$ be an orthogonal projection in $\R^n$ onto a $d$-dimensional subspace.
  Then
  $$
  \LL(PX, t \sqrt{d}) \le (Cp)^d.
  $$
\end{corollary}

This result can be regarded as a tensorization property of the concentration function. 
It will be deduced from Theorem~\ref{thm: dens PX} in Section~\ref{s: cor}.

\subsection{Anisotropic distributions}

Finally, we study concentration of anisotropic high-dimensional distributions, 
which take the form $AX$ for a fixed matrix $A$.
The key exponent that controls the behavior of the concentration function of $AX$
is the {\em stable rank} of $A$. We define it as
\begin{equation}							\label{eq: stable rank}
r(A)= \left \lfloor \frac{\norm{A}_\HS^2}{\norm{A}^2} \right \rfloor
\end{equation}
where $\|\cdot\|_\HS$ denotes the Hilbert-Schmidt norm.\footnote{This definition 
    differs slightly from the traditional definition of stable rank, 
    in which one does not take the floor function, i.e. where
    $r(A)= \norm{A}_\HS^2 / \norm{A}^2$.} 
Note that for any non-zero matrix, $1 \le r(A) \le \text{rank}(A)$.

\begin{theorem}[Concentration function of anisotropic distributions] \label{thm: conc AX}
  Consider a random vector $X= (X_1,\ldots, X_n)$ where
  $X_i$ are real-valued independent random variables.
  Let $t,p \ge 0$ be such that
  $$
  \LL(X_i, t) \le p \quad \text{for all } i=1,\ldots,n.
  $$ 
  Let $A$ be an $m \times n$ matrix and $\e \in (0,1)$. Then
  $$  
  \LL \big(AX, \, t\|A\|_\HS \big) \le (C_\e \, p)^{(1-\e) \, r(A)},
  $$
  where $C_\e$ depends only on $\e$.
\end{theorem}

A more precise version of this result is Theorem~\ref{thm: conc AX precise} and 
Corollary~\ref{cor: small ball} below. 
It will be deduced from Corollary~\ref{cor: conc PX} by replacing
$A$ by a dyadic sum of spectral projections. 

\begin{remark}[Scaling]
  To understand Theorem~\ref{thm: conc AX} better, 
  note that $\E \|AX\|_2^2 = \|A\|_\HS^2$
  if all $X_i$ have zero means and unit variances.
  This explains the scaling factor $\|A\|_\HS$ in Theorem~\ref{thm: conc AX}.
  Further, if $A$ is an orthogonal projection of rank $d$, then $\|A\|_\HS = \sqrt{d}$ 
  and $r(A) = d$, which recovers Corollary~\ref{cor: conc PX} in this case 
  up to $\e$ in the exponent. Moreover,  Theorem~\ref{thm: conc AX precise}
  and Corollary~\ref{cor: small ball-2} below will allow precise recovery, without any loss of $\e$. 
\end{remark}

\begin{remark}[Continuous distributions]
  In the particular case of continuous distributions, Theorem~\ref{thm: conc AX}
  states the following.
  Suppose the densities of $X_i$ are bounded by $K$.
  Then obviously $\LL(X_i, t) \le Kt$ for any $t \ge 0$, so Theorem~\ref{thm: conc AX} yields
  \begin{equation}				\label{eq: dens linear images}
  \LL(AX, t\|A\|_\HS) \le \left( C_\e \, t \right)^{(1-\e) \, r(A)}, \quad t \ge 0.
  \end{equation}
  A similar inequality was proved by Paouris \cite{Paouris} for random vectors $X$
  which satisfy three conditions: (a) $X$ is isotropic, i.e. all one-dimensional marginals of $X$
  have unit variance; (b) the distribution of $X$ is log-concave; 
  (c) all one-dimensional marginals are uniformly sub-gaussian.\footnote{Recall 
    that a random variable $Z$ is sub-gaussian if 
    $\P \{|Z|>t\} \le 2 \exp(-t^2/M^2)$ for all $t \ge 0$. The smallest $M \ge 0$ here can be taken as 
    a definition of the sub-gaussian norm of $Z$; see \cite{V tutorial}.}
  The inequality of Paouris states in this case that
  \begin{equation}				\label{eq: Paouris}
  \LL(AX, t\|A\|_\HS) \le \big(Ct \big)^{c \, r(A)}, \quad t \ge 0.
  \end{equation}
  Here $C$ is an absolute constant and $c \in (0,1)$ depends only on the 
  bound on the sub-gaussian norms.
  The distributions for which Paouris' inequality \eqref{eq: Paouris} applies 
  are not required to have independent coordinates. 
  On the other hand, the log-concavity assumption 
  for \eqref{eq: Paouris} is much stronger than 
  a uniform bound on the coordinate densities in \eqref{eq: dens linear images}.
\end{remark}

\begin{remark}[Large deviations]
  It is worthwhile to state here a related large deviation bound for $AX$ from \cite{RV HW}.
  If $X_i$ are independent, uniformly sub-gaussian, 
  and have zero means and unit variances, then  
  $$
  \Pr{ \big| \|AX\|_2 - \|A\|_\HS \big| \ge t \|A\|_\HS } \le 2 e^{-c t^2 r(A)}, \quad t \ge 0.
  $$
  Here $c>0$ depends only on the bound on the sub-gaussian norms of $X_i$.
\end{remark}

\subsection{The method}

Let us outline the proof of the key Theorem~\ref{thm: dens PX}, 
which implies all other results in this paper. 

A natural strategy would be to extend to higher dimensions
the simple one-dimensional argument leading to Theorem~\ref{thm: dens sums}, 
which was a combination of Ball's and Rogozin's theorems.
A higher-dimensional version of Ball's theorem is indeed available \cite{Ball 89};
it states that the maximal volume of a section of the cube by a subspace
of codimension $d$ is $(\sqrt{2})^d$.
However, we are unaware of any higher-dimensional versions 
of Rogozin's theorem \cite{Rogozin}.

An alternative approach to the special case of Theorem~\ref{thm: dens PX}
in dimension $d=1$ (and, as a consequence, to Corollary~\ref{cor: conc PX} in dimension $d=1$)
was developed in an unpublished
manuscript of Ball and Nazarov \cite{BN}. Although it does not achieve the optimal
constant $\sqrt{2}$ that appears in Theorem~\ref{thm: dens sums},
this approach avoids the delicate combinatorial arguments
that appear in the proof of Rogozin's theorem. 
The method of Ball and Nazarov is Fourier-theoretic; 
its crucial steps go back to Halasz~\cite{Halasz 75, Halasz 77} 
and Ball~\cite{Ball 86}. 

\smallskip

In this paper, we prove Theorem~\ref{thm: dens PX}
by generalizing the method of Ball and Nazarov \cite{BN} to higher dimensions
using Brascamp-Lieb inequality.
For educational purposes, we will start by presenting a version of 
Ball-Nazarov's argument in dimension $d=1$ in Sections~\ref{s: char functions} and
\ref{s: dim one}. The higher-dimensional argument will be presented in
Sections~\ref{s: all small}--\ref{s: proof}.

There turned out to be an unexpected difference between dimension $d=1$ 
and higher dimensions, which presents us with an an extra challenge.
The one-dimensional method works well under assumption that all coefficients $a_i$ are small,
e.g. $|a_i| \le 1/2$. The opposite case where there is a large coefficient $a_{i_0}$, is trivial;
it can be treated by conditioning on all $X_i$ except $X_{i_0}$. 

In higher dimensions, this latter case is no longer trivial. It corresponds
to the situation where some $\|Pe_{i_0}\|_2$ is large (here
$e_i$ denote the coordinate basis vectors).
The power of one random variable $X_{i_0}$ is not enough to yield
Theorem~\ref{thm: dens PX}; such argument would lead to a weaker bound $(C K \sqrt{d})^d$
instead of $(CK)^d$.

In Section~\ref{s: some large} we develop an alternative way to remove
the terms with large $\|Pe_i\|_2$ from the sum. It is based on a careful
tensorization argument for small ball probabilities.

\section{Deduction of Corollary~\ref{cor: conc PX} from Theorem~\ref{thm: dens PX}}  \label{s: cor}

We begin by recording a couple of elementary properties of concentration functions.

\begin{proposition}[Regularity of concentration function]			\label{prop: regularity}
  Let $Z$ be a random variable taking values in a $d$-dimensional subspace of $\R^n$.
  Then for every $M \ge 1$ and $t \ge 0$, we have
  $$
  \LL(Z,t) \le \LL(Z, Mt) \le (3M)^d \cdot \LL(Z,t).
  $$
\end{proposition}

\begin{proof}
The lower estimate is trivial. The upper estimate follows once we recall that 
a ball of radius $Mt$ in $\R^d$ can be covered by $(3M)^d$ balls of radius $t$.
\end{proof}

Throughout this paper, it will be convenient to work with the following equivalent definition of density.
For a random vector $Z$ taking values in a $d$-dimensional subspace $E$ of $\R^n$,
the density can be defined as
\begin{equation}         \label{eq: density}
f_Z(u) = \limsup_{t \to 0_+} \frac{1}{|B(t)|} \; \Pr{ \|Z - u\|_2 \le t }, \quad u \in E,
\end{equation} 
where $|B(t)|$ denotes the volume of a Euclidean ball with radius $t$ in $\R^d$.
Lebesgue differentiation theorem states that for random variable $Z$ with absolutely continuous distribution,
$f_Z(u)$ equals the actual density of $Z$
almost everywhere. 

The following elementary observation connects densities and concentration functions.

\begin{proposition}[Concentration function and densities]		\label{prop: sbp dens}
  Let $Z$ be a random variable taking values in a $d$-dimensional subspace of $\R^n$.
  Then the following assertions are equivalent:
  \begin{enumerate}[(i)]
    \item The density of $Z$ is bounded by $K^d$ almost everywhere;
    \item The concentration function of $Z$ satisfies
      $$
      \LL(Z, t \sqrt{d}) \le (Mt)^d \quad \text{for all } t \ge 0.
      $$
  \end{enumerate}
  Here $K$ and $M$ depend only on each other. In the implication (i) $\Rightarrow$ (ii),
  we have $M \le C K$ where $C$ is an absolute constant.
  In the implication (ii) $\Rightarrow$ (i), we have
  $K \le M$.
\end{proposition}

This proposition follows from the known bound
$t^d \le |B(t \sqrt{d})| \le (Ct)^d$ (see e.g. formula (1.18) in \cite{Pisier}). \qed 

\medskip

Now we are ready to deduce Corollary~\ref{cor: conc PX} from Theorem~\ref{thm: dens PX}.
The proof is a higher-dimensional version of the smoothing argument of Ball and Nazarov \cite{BN}.

\begin{proof}[Proof of Corollary~\ref{cor: conc PX}]
We can assume by approximation that $t>0$; then by rescaling (replacing $X$ with $X/t$) 
we can assume that $t=1$. Furthermore, 
translating $X$ if necessary, we reduce the problem to bounding
$\P\{ \|PX\|_2 \le \sqrt{d} \}$.
Consider independent random variables $Y_i$ uniformly distributed in $[-1/2,1/2]$,
which are also jointly independent of $X$.
We are seeking to replace $X$ by $X' := X+Y$.
By triangle inequality and independence, we have
\begin{align*}
\Pr{ \|PX'\|_2 \le 2\sqrt{d} }
&\ge \Pr{ \|PX\|_2 \le \sqrt{d} \text{ and } \|PY\|_2 \le \sqrt{d} }  \\
&= \Pr{ \|PX\|_2 \le \sqrt{d} } \; \Pr{ \|PY\|_2 \le \sqrt{d} }.
\end{align*}
An easy computation yields $\E \|PY\|_2^2 = d/12$, so
Markov's inequality implies that $\P\{ \|PY\|_2 \le \sqrt{d} \} \ge 11/12$.
It follows that
\begin{equation}         \label{eq: via PX'}
\Pr{ \|PX\|_2 \le \sqrt{d} } \le \frac{12}{11} \; \Pr{ \|PX'\|_2 \le 2\sqrt{d} }.
\end{equation}
Note that $X' = X+Y$ has independent coordinates whose densities
can be computed as follows:
$$
f_{X'_i}(u) = \Pr{ |X_i-u| \le 1/2 }
\le \LL(X_i,1/2).
$$
Applying Theorem~\ref{thm: dens PX}, we find that the density of
$PX'$ is bounded by $(CL)^d$, where $L = \max_i \LL(X_i,1/2)$.
Then Proposition~\ref{prop: sbp dens} yields that
$$
\Pr{ \|PX'\|_2 \le 2\sqrt{d} } \le \LL(PX', 2\sqrt{d}) \le (C_1 L)^d.
$$
Substituting this into \eqref{eq: via PX'}, we complete the proof.
\end{proof}

\begin{remark}[Flexible scaling in Corollary~\ref{cor: conc PX}]		\label{rem: flexible scaling}
 Using regularity of concentration function described in Proposition~\ref{prop: regularity},
 one can state the conclusion of Corollary~\ref{cor: conc PX} in a more flexible way:
 $$
 \LL(PX, Mt \sqrt{d}) \le (CMp)^d, \quad M \ge 1.
 $$
 We will use this observation later.  
\end{remark}

\section{Decay of characteristic functions}    \label{s: char functions}

We will now begin preparing our way for the proof of Theorem~\ref{thm: dens PX}. 
Our argument will use the following tail bound for the characteristic function 
$$
\phi_X(t) = \E e^{itX}
$$
of a random variable $X$ with bounded density. The estimate and its proof below are essentially 
due to Ball and Nazarov \cite{BN}. The reader can refer to \cite[Chapter 3]{LL} for 
the definition and basic properties of the decreasing rearrangements of functions. 

\begin{lemma}[Decay of characteristic functions]				\label{eq: char function decay}
  Let $X$ be a random variable whose density is bounded by $K$.
  Then the non-increasing rearrangement of the characteristic function of $X$ satisfies
  $$
  |\phi_X|^*(t) \le
  \begin{cases}
    1 - c(t/K)^2, & 0 < t < 2 \pi K \\
    \sqrt{2 \pi K/t}, & t \ge 2 \pi K.
  \end{cases}
  $$
\end{lemma}

\begin{proof}
The estimate for large $t$ will follow from Plancherel's identity.
The estimate for small $t$ will be based on a regularity argument
going back to Halasz \cite{Halasz 75}.

{\bf 1. Plancherel. }
By replacing $X$ with $KX$ we can assume that $K=1$.  
Let $f_X(\cdot)$ denote the density of $X$.
Thus $\phi_X(t) = \int_{-\infty}^\infty f_X(x) e^{itx} \, dx = \widehat{f_X}(-t/2\pi)$, 
according to the standard definition of the Fourier transform
\begin{equation}							\label{eq: Fourier}
\widehat{f}(t) = \int_{-\infty}^\infty f(x) e^{-2 \pi i t x} \, dx.
\end{equation}
By Plancherel's identity and using that $\|p_X\|_{L^1} = 1$ and $\|p_X\|_{L^\infty} \le K =1$, 
we obtain
$$
\|\phi_X\|_{L^2} = \sqrt{2\pi} \|p_X\|_{L^2}
\le \sqrt{2 \pi \|p_X\|_{L^1} \, \|p_X\|_{L^\infty}}
\le \sqrt{2\pi}.
$$
Chebychev's inequality then yields
\begin{equation}         \label{eq: phi t large}
|\phi_X|^*(t) \le \sqrt{\frac{2\pi}{t}}, \quad t > 0.
\end{equation}
This proves the second part of the claimed estimate. It remains to prove the first part.

\medskip

{\bf 2. Symmetrization.}
Let $X'$ denote an independent copy of $X$. Then
\begin{align*}
|\phi_X(t)|^2
&= \E e^{itX} \, \E \overline{e^{itX}}
= \E e^{itX} \, \E e^{-itX'}
= \E e^{it(X-X')} \\
&= \phi_{\tilde{X}}(t),
\quad \text{where } \tilde{X} := X-X'.
\end{align*}
Further, by symmetry of the distribution of $\tilde{X}$, we have
\[
\phi_{\tilde{X}}(t) = \E \cos(t |X|)
= 1 - 2 \E \sin^2 \big( \frac{1}{2} t |\tilde{X}| \big)
=: 1 - \psi(t).
\]
We are going to prove a bound of the form
\begin{equation}         \label{eq: mu psi}
\lambda\{ \tau: \psi(\tau) \le s^2 \} \le Cs, \quad 0 < s \le 1/2
\end{equation}
where $\lambda$ denotes the Lebesgue measure on $\R$.
Comnined with the identity $|\phi_X(t)|^2 = 1 - \psi(t)$, this bound would imply
$$
|\phi_X|^*(Cs) \le \sqrt{1-s^2} \le 1 - s^2/2, \quad 0 < s \le 1/2.
$$
Substituting $t=Cs$, we would obtain the desired estimate
$$
|\phi_X|^*(t) \le 1 - s^2/2C^2, \quad 0 < s \le C/2,
$$
which would conclude the proof (provided $C$ is chosen large enough so that $C/2 \ge 2\pi$).

{\bf 3. Regularity.}
First we observe that \eqref{eq: mu psi} holds for some fixed constant value of $s$.
This follows from the identity $|\phi_X(\tau)|^2 = 1 - \psi(\tau)$ and inequality \eqref{eq: phi t large}:
\begin{equation}         \label{eq: fixed t}
\lambda \big\{ \tau: \psi(\tau) \le \frac{1}{4} \big\}
= \lambda \{ \tau: |\phi_X(\tau)| \ge \sqrt{3/4} \}
\le 8\pi/3 \le 9.
\end{equation}
Next, the definition of $\psi(\cdot)$ and
the inequality $|\sin(mx)| \le m|\sin x|$ valid for $x \in \R$ and $m \in \N$ imply
that
$$
\psi(mt) \le m^2 \psi(t), \quad t > 0, \; m \in \N.
$$
Therefore
\begin{equation}         \label{eq: t discrete}
\lambda \big\{ \tau: \psi(\tau) \le \frac{1}{4m^2} \big\}
\le \lambda \big\{ \tau: \psi(m\tau) \le \frac{1}{4} \big\}
= \frac{1}{m} \, \lambda \big\{ \tau: \psi(\tau) \le \frac{1}{4} \big\}
\le \frac{9}{m},
\end{equation}
where in the last step we used \eqref{eq: fixed t}.
This establishes \eqref{eq: mu psi} for the discrete set of values $t = \frac{1}{2m}$, $m \in \N$.
We can extend this to arbitrary $t>0$ in a standard way,
by applying \eqref{eq: t discrete} for $m \in \N$ such that
$t \in (\frac{1}{4m}, \frac{1}{2m}]$.
This proves \eqref{eq: mu psi} and completes the proof of Lemma~\ref{eq: char function decay}.
\end{proof}

\section{Theorem~\ref{thm: dens PX} in dimension one. Densities of sums.}			\label{s: dim one}

Now we are going to give a ``soft'' proof of a version of Theorem~\ref{thm: dens sums}
due to Ball and Nazarov \cite{BN}. Their argument establishes Theorem~\ref{thm: dens PX} 
in dimension $d=1$. Let us state this result separately.

\begin{theorem}[Densities of sums]		\label{thm: dim one}
  Let $X_1,\ldots, X_n$ be real-valued independent random variables
  whose densities are bounded by $K$ almost everywhere.
  Let $a_1,\ldots,a_n$ be real numbers with $\sum_{j=1}^n a_j^2 = 1$.
  Then the density of $\sum_{j=1}^n a_j X_j$ is bounded by
  $CK$ almost everywhere.
\end{theorem}

\begin{proof}
By replacing $X_j$ with $KX_j$ we can assume that $K=1$.
By replacing $X_j$ with $-X_j$ when necessary we can assume that all $a_j \ge 0$.
We can further assume that $a_j > 0$ by dropping all zero terms from the sum.
If there exists $j_{0}$ with $a_{j_0} > 1/2$, then the conclusion follows by conditioning
on all $X_j$ except $X_{j_0}$. Thus we can assume that
$$
0 < a_j < \frac{1}{2} \quad \text{for all } j.
$$
Finally, by translating $X_j$ if necessary we reduce the problem to bounding the
density of $S = \sum_j a_j X_j$ at the origin.

We may assume that $\phi_{X_j} \in L_1$ by adding to $X_j$ 
an independent normal random variable with an arbitrarily small variance.
Fourier inversion formula associated with the Fourier transform \eqref{eq: Fourier} 
yields that the density of $S$ at the origin (defined using \eqref{eq: density})
can be reconstructed from its Fourier transform as 
\begin{equation}         \label{eq: char function}
f_S(0) = \int_\R \widehat{f_S}(x) \, dx
= \int_\R \phi_S(2 \pi x) \, dx
\le \int_\R |\phi_S(x)| \, dx =:I.
\end{equation}
By independence, we have
$\phi_S(x) = \prod_j \phi_{X_j} (a_j t)$, so
$$
I = \int_\R \prod_j |\phi_{X_j} (a_j x)| \, dx.
$$
We use the generalized H\"older's inequality with exponents $1/a_i^2$ whose reciprocals
sum to $1$ by assumption. It yields
\begin{equation}         \label{eq: I as product}
I \le \prod_j \Big( \int_\R |\phi_{X_j} (a_j x)|^{1/a_j^2} \, dx \Big)^{a_j^2}.
\end{equation}
The value of the integrals will not change if we replace the functions $|\phi_{X_j}|$ by their
non-increasing rearrangements $|\phi_{X_j}|^*$. After change of variable, we obtain
$$
I \le \prod_j \Big( \frac{1}{a_j} \int_0^\infty |\phi_{X_j}|^*(x)^{1/a_j^2} \, dx \Big)^{a_j^2}.
$$

We use Lemma~\ref{eq: char function decay} to bound the integrals
$$
I_j := \int_0^\infty |\phi_{X_j}|^*(x)^{1/a_j^2} \, dx
\le \int_0^{2\pi} (1-cx^2)^{1/a_j^2} \, dx + \int_{2\pi}^\infty (2\pi/x)^{1/(2a_j^2)} \, dx.
$$
Bounding $1-cx^2$ by $e^{-cx^2}$, we see that the first integral (over $[0,2\pi]$) is bounded
by $C a_j$. The second integral (over $[2\pi,\infty)$) is bounded by
$$
\frac{2\pi}{1/(2a_j^2) - 1} \le 8 \pi a_j^2,
$$
where we used that $a_j \le 1/2$. Therefore
$$
I_j \le C a_j + 8 \pi a_j^2 \le 2C a_j
$$
provided that constant $C$ is chosen large enough. Hence
$$
I \le \prod_j (2C)^{a_j^2} = (2C)^{\sum a_j^2} = 2C.
$$
Substituting this into \eqref{eq: char function} completes the proof.
\end{proof}

\section{Toward Theorem~\ref{thm: dens PX} in higher dimensions. The case of small $Pe_j$.}			\label{s: all small}

Our proof of Theorem~\ref{thm: dens PX} will go differently depending on
whether all vectors $Pe_j$ are small or some $Pe_j$ are large.
In the first case, we proceed with a high-dimensional version of the argument 
from Section~\ref{s: dim one}, where H\"older's inequality will be replaced by Brascamp-Lieb's inequality.
In the second case, we will remove the large vectors $Pe_j$ one by one, using a new precise tensorization 
property of concentration functions. 

In this section, we treat the case where all vectors $Pe_j$ are small.
Theorem~\ref{thm: dens PX} can be formulated in this case as follows.

\begin{proposition}				\label{prop: Pej small}
  Let $X$ be a random vector and $P$ be a projection
  which satisfy the assumptions of Theorem~\ref{thm: dens PX}.
  Assume that
  $$
  \|Pe_j\|_2 \le 1/2 \quad \text{for all } j=1,\ldots,n.
  $$
  Then the density of the random vector $PX$ is bounded by $(CK)^d$ almost everywhere.
\end{proposition}

The proof will be based on Brascamp-Lieb's inequality.

\begin{theorem}[Brascamp-Lieb \cite{BL}, see \cite{Ball 89}]
  Let $u_1,\ldots,u_n \in \R^d$ be unit vectors and $c_1,\ldots,c_n >0$ be real numbers satisfying
  $$
  \sum_{i=1}^n c_j u_j u_j^\tran = I_d.
  $$
  Let $f_1,\ldots,f_n : \R \to [0,\infty)$ be integrable functions. Then
  $$
  \int_{\R^n} \prod_{j=1}^n f_j(\< x, u_j\> )^{c_j} \; dx
  \le \prod_{j=1}^n \Big( \int_\R f_j(t) \; dt \Big)^{c_j}.
  $$
\end{theorem}

\begin{proof}[Proof of Proposition~\ref{prop: Pej small}]
The singular value decomposition of $P$ yields the existence of a
$d \times n$ matrix $R$ satisfying
$$
P = R^\tran R, \quad R R^\tran = I_d.
$$
It follows that $\|Px\|_2 = \|Rx\|_2$ for all $x \in \R^d$.
This allows us to replace $P$ by $R$ in the statement of the proposition.
Moreover, by replacing $X_j$ with $KX_j$ we can assume that $K=1$.
Finally, translating $X$ if necessary we reduce the problem to bounding
the density of $RX$ at the origin.

As in the proof of Theorem \ref{thm: dim one}, Fourier inversion formula 
associated with the Fourier transform in $n$ dimensions yields that the density of $RX$ at the origin
(defined using \eqref{eq: density})
can be reconstructed from its Fourier transform as
\begin{equation}         \label{eq: L1 char function}
f_{RX}(0) = \int_{\R^d} \widehat{f_{RX}}(x) \; dx
= \int_{\R^d} \phi_{RX}(2 \pi x) \; dx
\le \int_{\R^d} |\phi_{RX}(x)| \; dx
\end{equation}
where
\begin{equation}         \label{eq: phiRX}
\phi_{RX}(x) = \E \exp \big( i \< x, RX\> \big)
\end{equation}
is the characteristic function of $RX$.
Therefore, to complete the proof, it suffices to bound the integral in the
right hand side of \eqref{eq: L1 char function} by $C^d$.

In order to represent $\phi_{RX}(x)$
more conveniently for application of Brascamp-Lieb inequality,
we denote
$$
a_j := \|Re_j\|_2, \quad u_j := \frac{Re_j}{\|Re_j\|_2}.
$$
Then $R = \sum_{j=1}^n a_j u_j e_j^\tran$, so the identity $RR^\tran = I_d$ can be written as
\begin{equation}         \label{eq: uj decomposition}
\sum_{j=1}^n a_j^2 u_j u_j^\tran = I_d.
\end{equation}
Moreover, we have $\< x, RX\> = \sum_{i=1}^n a_j \< x, u_j\> X_j$.
Substituting this into \eqref{eq: phiRX} and using independence, we obtain
$$
\phi_{RX}(x) = \prod_{j=1}^n \E \exp \big( i a_j \< x, u_j \> X_j \big).
$$
Define the functions $f_1, \ldots, f_n : \R \to [0,\infty)$ as
$$
f_j(t) := \big| \E \exp(i a_j t X_j) \Big|^{1/a_j^2} = \big| \phi_{X_j} (a_j t) \big|^{1/a_j^2}.
$$
Recalling \eqref{eq: uj decomposition}, we apply Brascamp-Lieb inequality
for these functions and obtain
\begin{align}			\label{eq: int by product}
\int_{\R^d} |\phi_{RX}(x)| \; dx
  &= \int_{\R^d} \prod_{j=1}^n f_j \big( \< x, u_j\> \big)^{a_j^2} \; dx  \nonumber\\
  &\le \prod_{j=1}^n \Big( \int_\R f_j(t) \; dt \Big)^{a_j^2}
  = \prod_{j=1}^n \Big( \int_\R \big| \phi_{X_j} (a_j t) \big|^{1/a_j^2} \; dt \Big)^{a_j^2}.
\end{align}
We arrived at the same quantity as we encountered in one-dimensional argument in \eqref{eq: I as product}.
Following that argument, which uses the assumption that all $a_j \le 1/2$, we bound the product
the quantity above by
$$
(2C)^{\sum_{j=1}^n a_j^2}.
$$
Recalling that $a_j = \|Re_j\|_2$ and , we find that
$\sum_{j=1}^n a_j^2 = \sum_{j=1}^n \|Re_j\|_2^2 = \tr(RR^\tran) = \tr(I_d) = d$.
Thus the right hand side of \eqref{eq: int by product} is bounded by $(2C)^d$.
The proof of Proposition~\ref{prop: Pej small} is complete.
\end{proof}

\section{Toward Theorem~\ref{thm: dens PX} in higher dimensions. Removal of large $Pe_j$.}				\label{s: some large}

Next we turn to the case where not all vectors $Pe_i$ are small.
In this case, we will remove the large vectors $Pe_i$ one by one.
The non-trivial task is how not to lose power at each step.
This will be achieved with the help of the following 
precise tensorization property of small ball probabilities.

\begin{lemma}[Tensorization]				\label{lem: tensorization}
  Let $Z_1, Z_2 \ge 0$ be random variables and $M_1, M_2, p \ge 0$ be real numbers.
  Assume that
  \begin{enumerate}[(i)]
    \item $\Pr{Z_1 \le t \;|\; Z_2} \le M_1 t$ almost surely in $Z_2$ for all $t \ge 0$;
    \item $\Pr{Z_2 \le t} \le M_2 t^p$ for all $t \ge 0$.
  \end{enumerate}
  Then
  $$
  \Pr{\sqrt{Z_1^2 + Z_2^2} \le t} \le \frac{C M_1 M_2 }{\sqrt{p+1}} \; t^{p+1} \quad \text{for all } t \ge 0.
  $$
\end{lemma}

\begin{remark}
  This lemma will be used later in an inductive argument.
  To make the inductive step, two features will be critical:
  (a) the term of order $\sqrt{p}$ in the denominator
  of the probability estimate; 
  (b) the possibility of choosing different values for the parameters $M_1$ and $M_2$.
\end{remark}

\begin{proof}
Denoting $s = t^2$, we compute the probability by iterative integration in $(Z_1^2, Z_2^2)$ plane:
\begin{equation}         \label{iterative integral}
\Pr{Z_1^2 + Z_2^2 \le s} = \int_0^s \Pr{Z_1 \le (s-x)^{1/2} \;|\; Z_2^2 = x} \; d F_2(x)
\end{equation}
where $F_2(x) = \Pr{Z_2^2 \le x}$ is the cumulative distribution function of $Z_2^2$.
Using hypothesis (i) of the lemma, we can bound the right hand side of \eqref{iterative integral} by
$$
M_1 \int_0^s (s-x)^{1/2} \; d F_2(x) = \frac{M_1}{2} \int_0^s F_2(x) (s-x)^{-1/2} \; dx,
$$
where the last equation follows by integration by parts.
Hypothesis (ii) of the lemma says that $F_2(x) \le M_2 \, x^{p/2}$, so the expression above
is bounded by
$$
\frac{M_1 M_2}{2} \int_0^s x^{p/2} (s-x)^{-1/2} \; dx
= \frac{M_1 M_2}{2} \, s^{\frac{p+1}{2}} \int_0^1 u^{p/2} (1-u)^{-1/2} \; du
$$
where the last equation follows by substitution $x=su$.

The integral in the right hand side is the value of beta function
$$
B \Big( \frac{p}{2}+1, \frac{1}{2} \Big).
$$
Bounding this value is standard. One can use the fact that
$$
B(x,y) \sim \Gamma(y) x^{-y} \quad \text{as } x \to \infty, \quad y \text{ fixed}
$$
which follows from the identity $B(x,y) = \Gamma(x) \Gamma(y) / \Gamma(x+y)$
and Stirling's approximation. (Here $f(x) \sim g(x)$ means that $f(x) / g(x) \to 1$.)
It follows that
$$
B \Big( \frac{p}{2}+1, \frac{1}{2} \Big)
\sim \Gamma \Big(\frac{1}{2}\Big) \Big( \frac{p}{2}+1 \Big)^{-1/2} \quad \text{as } p \to \infty.
$$
Therefore the ratio of the left and right hand sides is bounded in $p$.
Hence there exists an absolute constant $C$ such that
$$
B \Big( \frac{p}{2}+1, \frac{1}{2} \Big) \le \frac{C}{\sqrt{p+1}} \quad \text{for all } p \ge 0.
$$

We have proved that
$$
\Pr{Z_1^2 + Z_2^2 \le s} \le \frac{M_1 M_2}{2} \, s^{\frac{p+1}{2}} \frac{C}{\sqrt{p+1}}.
$$
Substituting $s=t^2$ finishes the proof.
\end{proof}

\begin{corollary}[Tensorization, continued]					\label{cor: tensorization}
  Let $Z_1, Z_2 \ge 0$ be random variables and $K_1, K_2 \ge 0$, $d > 1$ be real numbers.
  Assume that
  \begin{enumerate}[(i)]
    \item $\Pr{Z_1 \le t \;|\; Z_2} \le K_1 t$ almost surely in $Z_2$ for all $t \ge 0$;
    \item $\Pr{Z_2 \le t \sqrt{d-1}} \le (K_2 t)^{d-1}$ for all $t \ge 0$.
  \end{enumerate}
  Then
  $$
  \Pr{\sqrt{Z_1^2 + Z_2^2} \le t \sqrt{d}} \le (K_2 t)^d \quad \text{for all } t \ge 0,
  $$
  provided that $K_1 \le c K_2$ with a suitably small absolute constant $c$. 
\end{corollary}

\begin{proof}
Random variables $Z_1$, $Z_2$ satisfy the assumptions of Lemma~\ref{lem: tensorization} with
$$
M_1 = K_1, \quad M_2 = \Big( \frac{K_2}{\sqrt{d-1}} \Big)^{d-1}, \quad p = d-1.
$$
The conclusion of that lemma is that
$$
\Pr{\sqrt{Z_1^2 + Z_2^2} \le t \sqrt{d}}
\le C K_1 \Big( \frac{K_2}{\sqrt{d-1}} \Big)^{d-1} \frac{1}{\sqrt{d}} (t \sqrt{d})^d.
$$
Using the hypothesis that $K_1 \le c K_2$, we bound the right hand side by
$$
C c (K_2 t)^d \, \Big( \frac{d}{d-1} \Big)^{\frac{d-1}{2}} \le 3 C c (K_2 t)^d.
$$
If we choose $c = 1/(3C)$ then the right hand side gets bounded by $(K_2 t)^d$, as claimed.
\end{proof}

\begin{proposition}[Removal of large $Pe_i$]				\label{prop: Pei large}			
  Let $X$ be a random vector satisfying the assumptions of Theorem~\ref{thm: dens PX},
  and let $P$ be an orthogonal projection in $\R^n$ onto a $d$-dimensional subspace.
  Let $\nu >0$, and assume that there exists $i \in \{1,\ldots,n\}$ such that
  $$
  \|Pe_i\|_2 \ge \nu.
  $$
  Define $Q$ to be the orthogonal projection in $\R^n$ such that
  $$
  \ker(Q) = \Span\{ \ker(P), P e_i \}.
  $$
  Let $M \ge C_0$ where $C_0$ is an absolute constant. If
  \begin{equation}         \label{eq: QX hypothesis}
  \Pr{ \|QX\|_2 \le t \sqrt{d-1} } \le (M Kt/\nu)^{d-1} \quad \text{for all } t \ge 0,
  \end{equation}
  then
  $$
  \Pr{ \|PX\|_2 \le t \sqrt{d} } \le (M Kt/\nu)^d \quad \text{for all } t \ge 0.
  $$
\end{proposition}

\begin{proof}
Without loss of generality, we can assume that $i=1$.
Let us record a few basic properties of $Q$.
A straightforward check shows that 
\begin{equation}							\label{eq: P-Q}
\text{$P-Q$ is the orthogonal projection onto $\Span(Pe_1)$.}
\end{equation}
It follows that $(P-Q)e_1 = Pe_1$, since the orthogonal projection of $e_1$ 
onto $\Span(Pe_1)$ equals $Pe_1$. Canceling $Pe_1$ on both sides, we have
\begin{equation}							\label{eq: Qe1}
Qe_1 = 0.
\end{equation}
It follows from \eqref{eq: P-Q} that $P$ has the form 
\begin{equation}         \label{eq: Px}
Px = \Big( \sum_{j=1}^n a_j x_j \Big) Pe_1 + Q x \quad \text{for } x = (x_1,\ldots,x_n) \in \R^n,
\end{equation}
where $a_j$ are fixed numbers (independent of $x$).
Substituting $x=e_1$, we obtain using \eqref{eq: Qe1} that $Pe_1 = a_1 Pe_1 + Qe_1 = a_1 Pe_1$. 
Thus 
\begin{equation}         \label{eq: a1}
a_1 = 1.
\end{equation}
Furthermore, we note that
\begin{equation}         \label{eq: Qx x1}
\text{$Qx$ does not depend on $x_1$}
\end{equation}
since $Qx = Q(\sum_{i=1}^n x_j e_j) = \sum_{i=1}^n x_j Qe_j$ and
$Qe_1 = 0$ by \eqref{eq: Qe1}.
Finally, since $Pe_1$ is orthogonal to the image of $Q$, 
the two vectors in the right side of \eqref{eq: Px}
are orthogonal. Thus
\begin{equation}         \label{eq: Px norm}
\|Px\|_2^2 = \Big( \sum_{j=1}^n a_j x_j \Big)^2 \|Pe_1\|_2^2 + \|Qx\|_2^2.
\end{equation}

Now let us estimate $\|PX\|_2$ for a random vector $X$.
We express $\|PX\|_2^2$ using \eqref{eq: Px norm} and \eqref{eq: a1} as
$$
\|P X\|_2^2
= \Big( X_1 + \sum_{j=2}^n a_j X_j \Big)^2 \|Pe_1\|_2^2 + \|QX\|_2^2 \\
=: Z_1^2 + Z_2^2
$$
and try to apply Corollary~\ref{cor: tensorization}. Let first us check the hypotheses of that corollary.
Since by \eqref{eq: Qx x1} $Z_2$ is determined by $X_2,\ldots,X_n$ (and is independent of $X_1$),
and $\|Pe_i\|_2 \ge \nu$ by a hypothesis of the lemma, we have
\begin{align*}
\Pr{Z_1 \le t \;|\; Z_2}
  &\le \max_{X_2,\ldots,X_n} \Pr{ \Big| X_1 + \sum_{j=2}^n a_j X_j \Big| \le t/\nu \;\Big|\; X_2,\ldots,X_n} \\
  &\le \max_{u \in \R} \Pr{|X_1-u| \le t/\nu}
  \le K t/\nu.
\end{align*}
The last inequality follows since the density of $X_1$ is bounded by $K$.
This verifies hypothesis (i) of Corollary~\ref{cor: tensorization}
with $K_1 = K/\nu$. Hypothesis (ii) follows immediately from \eqref{eq: QX hypothesis},
with $K_2 = MK/\nu$. If $M \ge 1/c =: C_0$ then $K_1 \le cK_2$ as required
in Corollary~\ref{cor: tensorization}. It yields
$$
\Pr{\sqrt{Z_1^2 + Z_2^2} \le t \sqrt{d} } \le (MKt/\nu)^d \quad \text{for all } t \ge 0.
$$
This completes the proof.
\end{proof}

\section{Theorem~\ref{thm: dens PX} in higher dimensions: completion of the proof}			\label{s: proof}

Now we are ready to prove Theorem~\ref{thm: dens PX}.
Replacing $X_j$ with $KX_j$ we can assume that $K=1$.
By Proposition~\ref{prop: sbp dens}, it suffices to bound the concentration function as follows:
$$
\LL(PX, t \sqrt{d}) \le (C_1 t)^d, \quad t \ge 0,
$$
where $C_1$ is an absolute constant.
Translating $X$ if necessary, we can reduce the problem to showing that
\begin{equation}         \label{eq: induction}
\Pr{ \|PX\|_2 \le t \sqrt{d} } \le (C_1 t)^d, \quad t \ge 0.
\end{equation}
We will prove this by induction on $d$.
Choose $C_1$ a sufficiently large constant, depending on the value of the constants 
denoted by $C$ in Theorem~\ref{thm: dim one}, Proposition~\ref{prop: sbp dens}, 
Proposition~\ref{prop: Pej small}, 
and the constant denoted by $C_0$ in Proposition~\ref{prop: Pei large}.

For the base of the induction, inequality \eqref{eq: induction} for $d=1$ 
follows from Theorem~\ref{thm: dim one}.

Assume that the statement \eqref{eq: induction} holds in dimension $d-1 \in \N$,
so one has
\begin{equation}         \label{eq: induction hypothesis}
\Pr{ \|QX\|_2 \le t \sqrt{d-1} } \le (C_1 t)^{d-1}, \quad t \ge 0
\end{equation}
for the projection $Q$ onto any $(d-1)$-dimensional subspace of $\R^n$.
We would like to make an induction step, i.e. prove \eqref{eq: induction} in dimension $d$.

If $\|Pe_i\|_2 < 1/2$ for all $i \in [n]$, then \eqref{eq: induction}
follows from Proposition~\ref{prop: Pej small} together with Proposition~\ref{prop: sbp dens}.
Alternatively, if there exists $i \in [n]$ such that $\|Pe_i\|_2 \ge 1/2$,
we can apply Proposition~\ref{prop: Pei large} with $M = C_1/2$.
Moreover, since the rank of $Q$ is $d-1$,
the assumption \eqref{eq: QX hypothesis} is also satisfied due to the induction hypothesis
\eqref{eq: induction hypothesis} and the choice of $M$.
Then an application of Proposition~\ref{prop: Pei large} yields \eqref{eq: induction}.
This completes the proof of Theorem~\ref{thm: dens PX}.
\qed

\section{Theorem~\ref{thm: conc AX}. Concentration of anisotropic distributions.}

In this section we prove a more precise version of Theorem~\ref{thm: conc AX} for 
random vectors of the form $AX$, where $A$ is a fixed $m \times n$ matrix. 

The singular values of $A$ arranged in a non-increasing order are denoted by $s_j(A)$, $j = 1, \ldots, m \wedge n$.
To simplify the notation, we set $s_j(A)=0$ for $j > m \wedge n$, and we will do the same for singular vectors of $A$.

The definition of the {\em stable rank} of $A$ from \eqref{eq: stable rank} reads as 
$$
r(A)= \left \lfloor \frac{\norm{A}_\HS^2}{\norm{A}^2} \right \rfloor 
= \left \lfloor \frac{\sum_{j=1}^\infty s_j(A)^2}{s_1(A)^2} \right \rfloor.
$$
To emphasize the difference between the essential rank and the rank, we set 
\begin{equation}				\label{eq: delta}
\d(A)= \frac{\sum_{j=r(A)+1}^\infty s_j(A)^2}{\sum_{j=1}^\infty s_j(A)^2}.
\end{equation}
Thus $0 \le \d(A) \le 1$.
Note that the numerator in \eqref{eq: delta} 
is the square of the distance from $A$ to the set of matrices of rank at most $r(A)$ 
in the Hilbert-Schmidt metric, while the denominator equals $\|A\|_\HS^2$. 
In particular, $\d(A)=0$ if and only if $A$ is an orthogonal projection up to an isometry;
in this case $r(A) = \rank(A)$. 

\begin{theorem} 		\label{thm: conc AX precise}
  Consider a random vector $X= (X_1,\ldots, X_n)$ where
  $X_i$ are real-valued independent random variables.
  Let $t,p \ge 0$ be such that
  $$
  \LL(X_i, t) \le p \quad \text{for all } i=1,\ldots,n.
  $$ 
  Then for every $m \times n$ matrix $A$ and for every $M \ge 1$ we have
  \begin{equation} \label{i: all delta}
    \LL(AX, Mt\|A\|_\HS) \le \left( CMp/\sqrt{\d(A)} \right)^{r(A)}.
 \end{equation}
 provided $\d(A)>0$. Moreover, if $\d(A) < 0.4$, then
 \begin{equation} \label{i: small delta}
  \LL(AX, Mt\|A\|_\HS) \le \left( CMp \right)^{r_0(A)},
 \end{equation} 
 where $r_0(A)=\lceil (1-2\d(A))r(A) \rceil$. 
\end{theorem}

\begin{remark}[Orthogonal projections]
  For orthogonal projections we have $\d(A)=0$, $r_0(A) = r(A) = \rank(A)$, 
  so the second part of Theorem~\ref{thm: conc AX precise} recovers 
  Corollary~\ref{cor: conc PX}.
\end{remark}

\begin{remark}[Stability]
  Note the $\lceil \cdot \rceil$ instead of $\lfloor \cdot \rfloor$ 
  in the definition of $r_0(A)$. This offers some stability of the concentration function. 
  Indeed, a small perturbation of a $d$-dimensional 
  orthogonal projection will not change the exponent $r_0(A) = r(A) = d$ 
  in the small ball probability.
\end{remark}

\begin{remark}[Flexible scaling]
  The parameter $M$ offers some flexible scaling, which may be useful in applications. 
  For example, knowing that $\LL(X_i,t)$ are all small, Theorem~\ref{thm: conc AX precise}
  allows one to bound $\LL(AX, 10t\|A\|_\HS)$
  rather than just $\LL(AX, t\|A\|_\HS)$. Note that such result would not trivially follow by applying 
  Remark~\ref{rem: flexible scaling}.
\end{remark}

\medskip

\begin{proof}[Proof of Theorem~\ref{thm: conc AX precise}]
We will first prove an inequality that is more general than \eqref{i: all delta}.
Denote
$$
S_r(A)^2 = \sum_{j=r+1}^\infty s_j(A)^2, \quad r =0, 1, 2, \ldots
$$
Then, for every $r$, we claim that 
\begin{equation}				\label{eq: all delta general}
\LL(AX, M t S_r(A)) \le (CMp)^r.
\end{equation}
This inequality would imply \eqref{i: all delta} by rescaling, since $S_r(A) = \sqrt{\d(A)} \|A\|_\HS$.

\medskip

Before we prove \eqref{eq: all delta general}, let us make some helpful reductions. 
First, by replacing $X$ with $X/t$ we can assume that $t=1$. 
We can also assume that the vector $u$ appearing the 
definition \eqref{eq: concentration function} of the concentration function 
$\LL(AX, M t S_r)$ equals zero; 
this is obvious by first projecting $u$ onto the image of $A$ and then appropriately 
translating $X$. With these reductions, the claim \eqref{eq: all delta general} becomes 
\begin{equation}							\label{eq: all delta general reduced}
\Pr{ \norm{AX}_2 \le M S_r(A) } \le (CMp)^r.
\end{equation}

Let $A= \sum_{j=1}^\infty s_j(A) u_j v_j^\tran$ be the singular value decomposition of $A$.
For $l=0,1,2,\ldots$, consider the spectral projections $P_l$
defined as  
$$
P_0=  \sum_{j=1}^{r} v_j v_j^\tran \quad \text{ and } \quad
P_l= \sum_{j=2^{l-1} r+1}^{2^l r} v_j v_j^\tran, \quad l=1,2,\ldots
$$
Note that $\text{rank}(P_0)=r$ and $\text{rank}(P_l) = 2^{l-1} r$ for $l=1,2,\ldots$

We shall bound $\|AX\|_2$ below and $S_r(A)$ above and then compare the two estimates.
First, using the monotonicity of the singular values, we have
$$
\norm{AX}_2^2
= \sum_{j=1}^\infty s_j(A)^2 \< X, v_j \> ^2
\ge \sum_{l=0}^\infty s_{2^l r}(A)^2 \norm{P_l X}_2^2.
$$
Next, again by monotonicity, 
$$
S_r(A)^2 \le \sum_{l=0}^\infty 2^l r \cdot s_{2^l r}(A)^2. 
$$
Comparing these two estimates term by term, we obtain  
\begin{equation}				\label{eq: Sr sum}
\Pr{ \norm{AX}_2 < M S_r(A) }
\le \sum_{l=0}^\infty \Pr{ \norm{P_l X}_2^2 < M^2 \, 2^l r }.
\end{equation}
Applying Corollary~\ref{cor: conc PX} (see Remark~\ref{rem: flexible scaling})
and noting that $2^l r\le 2 \rank(P_l)$, we find that
\begin{equation}				\label{eq: PIX}
\Pr{ \norm{P_l X}_2^2 < M^2 \, 2^l r } \le (C_0Mp)^{2^l r}, \quad l=0,1,2,\ldots
\end{equation}
where $C_0$ is an absolute constant.
We will shortly conclude that \eqref{eq: all delta general reduced} holds with $C = 10 C_0$.
Without loss of generality we can assume that $CMp \le 1$, so  $C_0Mp \le 1/10$.
Thus upon substituting \eqref{eq: PIX} into \eqref{eq: Sr sum} we obtain a convergent series
whose sum is bounded by $(10 C_0 Mp)^r$. This proves \eqref{eq: all delta general reduced}.

\medskip

We now turn to proving \eqref{i: small delta}. 
By replacing $X$ with $X/t$ and $A$ with $A/\|A\|$ we can assume that $t=1$ and $\|A\|=1$,
and reduce our task to showing that  
\begin{equation}							\label{eq: small delta reduced}
\Pr{ \|AX\|_2 \le M\|A\|_\HS } \le (CMp)^{r_0(A)}.
\end{equation}

For shortness, denote $r = r(A) = \lfloor \|A\|_\HS^2 \rfloor$ and $\d = \d(A)$. 
Set $k = \lfloor (1-2\d)r \rfloor$. We claim that 
\begin{equation}				\label{eq: claim sk}
s_{k+1}(A) \ge \frac{1}{2}.
\end{equation}
Assume the contrary. By definition of $\d$ and $r$, we have
\begin{equation}				\label{eq: sum lower}
\sum_{j=1}^r s_j(A)^2 = (1-\d) \|A\|_\HS^2 \ge (1-\d) r.
\end{equation}
On the other hand, by our assumption and monotonicity, 
$s_j(A)$ are bounded by $\|A\|=1$ for all $j$, and by $1/2$ for $j \ge k+1$. Thus
\begin{equation}				\label{eq: sum above}
\sum_{j=1}^r s_j(A)^2 \le \sum_{j=1}^k 1^2 + \sum_{j=k+1}^r \Big( \frac{1}{2} \Big)^2
= k + (r-k) \frac{1}{4}.
\end{equation}
Since the expression in the right hand side increases in $k$ and $k \le (1-2\d)r$, 
we can further bound the sum in \eqref{eq: sum above} by $(1-2\d) r + 2\d r \cdot \frac{1}{4}$.
But this is smaller than $(1-\d)r$, the lower bound for the same sum in \eqref{eq: sum lower}.
This contradiction establishes our claim \eqref{eq: claim sk}.

Similarly to the first part of the proof, we shall bound $\|AX\|_2$ below and $S_r(A)$ above 
and then compare the two estimates. For the lower bound, we consider the spectral projection 
$$
Q_{k+1} = \sum_{j=1}^{k+1} v_j v_j^\tran
$$
and using \eqref{eq: claim sk}, we bound
$$
\|AX\|_2^2 
\ge s_{k+1}(A)^2 \|Q_{k+1}X\|_2^2 
\ge \frac{1}{4} \|Q_{k+1}X\|_2^2.
$$
For the upper bound, we note that since $\d \le 0.4$ by assumption 
and $r = \lfloor \|A\|_\HS^2 \rfloor \ge 1$, we have
$$
\|A\|_\HS^2 \le 2r \le 10(1-2\d)r \le 10(k+1)
$$
Applying Corollary~\ref{cor: conc PX} (see Remark~\ref{rem: flexible scaling}) and recalling that
$\rank(Q_k) = k+1$, we conclude that
$$
\Pr{ \|AX\|_2 \le M\|A\|_\HS } 
\le \Pr{ \|Q_{k+1}X\|_2^2 \le 40 M^2 (k+1) } 
\le (CMp)^{k+1}.
$$
It remains to note that $k+1 \ge r_0(A)$ by definition. 
This establishes \eqref{eq: small delta reduced} whenever $CMp \le 1$.
The case when $CMp >1$ is trivial. 
Theorem~\ref{thm: conc AX precise} is proved.
\end{proof}

Theorem~\ref{thm: conc AX precise} implies the following more precise version 
of Theorem~\ref{thm: conc AX}.

\begin{corollary} \label{cor: small ball}
  Consider a random vector $X= (X_1,\ldots, X_n)$ where
  $X_i$ are real-valued independent random variables.
  Let $t,p \ge 0$ be such that
  $$
  \LL(X_i, t) \le p \quad \text{for all } i=1,\ldots,n.
  $$ 
  Then for every $m \times n$ matrix $A$, every $M \ge 1$ and every $\e \in (0,1)$ we have
  $$
  \LL(AX, Mt\|A\|_\HS) \le \left( C_\e M p \right)^{\lceil (1-\e) r(A)\rceil}.
  $$
  Here $C_\e = C/\sqrt{\e}$ and $C$ is an absolute constant.
\end{corollary}

\begin{proof}
The result follows from Theorem~\ref{thm: conc AX precise} where we apply estimate \eqref{i: all delta} 
whenever $\d(A) \ge \e/2$ and \eqref{i: small delta} otherwise. 
\end{proof}

In the case when the densities of $X_i$ are uniformly bounded by $K>0$, Corollary~\ref{cor: small ball} yields
\begin{equation}  \label{eq: cont density}
  \LL(AX, t \|A\|_\HS) \le \left( \frac{CKt}{\sqrt{\e}} \right)^{\lceil (1-\e) r(A)\rceil} \quad \text{for all } t,\e>0.
\end{equation}
Applying this estimate with $\e= \frac{1}{2 r(A)}$ and $\tau= t \sqrt{r(A)}$, we derive  a bound on the Levy concentration function, which is lossless in terms of power of the small ball radius. Such bound may be useful for estimating the negative moments of the norm of $AX$.

\begin{corollary} \label{cor: small ball-2}
 Let $X=(X_1 \etc X_n)$ be a vector with independent random coordinates. Assume that the densities on $X_1, \ldots, X_n$ are bounded by $K$.
 Let $A$ be an $m \times n$ matrix.
 Then for any  $\tau>0$,
 \[
   \LL(AX, \tau \norm{A}) \le \left(CK \tau \right)^{ r(A)}.
 \]
\end{corollary}

\end{document}